\documentclass[12pt]{article}

\usepackage[english]{babel}

\usepackage[english]{babel}
\usepackage[dvips, final]{graphicx}
\usepackage[T1]{fontenc} 
\usepackage{amsmath}
\usepackage{amsfonts}
\usepackage{amssymb}
\usepackage{amsthm}
\usepackage{enumerate}
\usepackage{color}
\usepackage{epsfig}

\usepackage{makeidx}
\usepackage{xspace}
\usepackage{tabularx}
\usepackage{verbatim}
\usepackage[14pt]{extsizes}
\usepackage{geometry}
\usepackage{texnames}
\usepackage{fancyhdr}
\usepackage{cite}
\graphicspath{{:fig:}}

\RequirePackage{amsthm}
\theoremstyle{plain} 

\usepackage{amsmath, amsthm, amscd, amsfonts, amssymb, graphicx, color}

\newtheorem{theorem}{Theorem}[section]
\newtheorem{lemma}[theorem]{Lemma}

\newtheorem{corollary}[theorem]{Corollary}
\theoremstyle{definition}

\theoremstyle{remark}
\newtheorem{remark}[theorem]{Remark}
\numberwithin{equation}{section}
\usepackage{tabularx}

\usepackage{cite}

\usepackage{dcolumn}
\usepackage{bm}

\theoremstyle{remark}

\def\@tocrmarg{3em}
\def\@pnumwidth{2em}

\begin{document}

\noindent\large{\bf Solution to the  Volterra Matrix  Equation\\ of the 1st kind
with Piecewise Continuous Kernels}\footnote{This work is partly supported by RFBR, project No. 11-08-00109,  Deutscher Akademischer Austauschdienst (DAAD), No.~A1200665. 
It is carried out within Federal Framework Programm ``Scientific and Scientific-Academic Staff of Innovative Russia'' within project No. 2012-1.2.2-12-000-1001-012.}\\

\normalsize
\noindent {\bf Denis Sidorov}\\

In this paper we address the
following matrix Volterra integral equation of the first kind
\begin{equation}
\int\limits_0^t K(t,s) x(s) ds = f(t), \, 0<t \leq T.
\label{eq1}
\end{equation}
We assume  $m \times m$ matrix kernel $K(t,s)$ 
 has jump discontinuities on curves $s=\alpha_i(t), \, i=1, ... , n-1$,
$0 \leq s \leq t \leq T$ as follows
 \begin{equation}
    K(t,s) = \left\{ \begin{array}{ll}
         \mbox{$K_1(t,s), \,\,\,\,\,\,\,\,\,\,\,\,\, {0} \leq s \leq \alpha_1(t)$}, \\
         \mbox{$K_2(t,s), \,\,\, \alpha_1(t) < s \leq \alpha_2(t)$}, \\
         \mbox{\,\,\,\,\dots \,\,\,\,\,\,\,\,\,\, \,\, \dots \dots} \\
         \mbox{$K_n(t,s), \,\,\, \alpha_{n-1}(t) < s \leq  { t}$}, \\
        \end{array} \right. \label{eq2} 
         \end{equation}
        $f(t) = (f_1(t), \dots , f_m(t))^{\prime},$
$x(t) = (x_1(t), \dots , x_m(t))^{\prime}.$
Matrices $K_i(t,s)$ are $m \times m.$ We assume that these matrices are
defined, continuous and have continuous derivatives wrt $t$ 
correspondingly in 
$D_i = \bigl\{ s,t | \alpha_{i-1}(t) < s \leq \alpha_i(t) \bigr\},$
$i=\overline{1,n}, $ $\alpha_0=0, $  $\alpha_n(t) = t.$
 Functions $f_i(t), \alpha_i(t)$ have continuous derivatives,
$f_i(0)=0,$
$\alpha_i(0)=0,$  $0< \alpha_1^{\prime}(0) < \alpha_2^{\prime}(0) < \cdots < \alpha_{n-1}^{\prime}(0) <1,$ 
$0< \alpha_1(t) < \alpha_2(t) < \cdots < \alpha_{n-1}(t) <t$
for $t \in (0,T].$
In this chapter we search for the continuous solutions
of the equation \eqref{eq1}
for $t \in (0,T^{\prime}],$ where $0< T^{\prime} \leq T,$
$\lim\limits_{t\rightarrow +0} x(t)$ can also be infinite.
We assume below that each matrix $K_i(t,s), \, i=\overline{1,n}$
has continuously differentiable wrt $t$ extension on the interval $0\leq s \leq t \leq T$.
Homogeneous systems can have nontrivial solutions. 
Differentiation of the system\eqref{eq1} wrt $t$ in contrast to the classic
case \cite{lit6, lit9, lit13} lead us to the new class of the Volterra integral
equations with functionally perturbed argument. That is the reason why the 
conventional techniques are not applicable for the system \eqref{eq1} 
and studies of the systems with jump discontinuous kernels 
\eqref{eq2} are theoretically and practically interesting object to study.
Operator equations with functionally perturbed argument are addressed in  \cite{lit2, lit12}.   

The paper is organized as follows.
In section 1 we propose the algorithm for construction of the logarithmic
power asymptotics
\begin{equation}
\hat{x}(t) = \sum\limits_{i=0}^N x_i(\ln t) t^i
\label{eq3}
\end{equation}
of the desired continuous solutions of the system \eqref{eq1}. 
In the section 2 we proof theorem of existance of the parametric families
of solutions to the system \eqref{eq1}. Finally we derive
the sufficient conditions for existance and uniqueness of continuous 
solution.

\section{Construction of asymptotic approximation
of solutions to inhomogeneous system }

Let the following condition be fuilfilled
\begin{enumerate} [{\bf (A)}]
\item Matrices $\mathcal{P}_i = \sum\limits_{\nu+\mu = 1}^N K_{i \nu \mu} t^{\nu} s^{\mu}, \, i=\overline{1,n}$ exist,
vector-function $f^N(t) = \sum\limits_{\nu=1}^N f_{\nu} t^{\nu},$
and polynomials $\alpha_i^N(t) = \sum\limits_{\nu = 1}^N \alpha_{i \nu} t^{\nu}, i=\overline{1,n-1},$
where
$0< \alpha_{11} < \alpha_{21}  < \dots < \alpha_{n-1,1} <1,$
such as for $t\rightarrow +0,$ $s\rightarrow +0$ the following estimates are fulfilled $ K_i(t,s) -\mathcal{P}_i(t,s) = \mathcal{O}((t+s)^{N+1}), \, i=\overline{1,n},$ 
$f(t) - f^N(t) = \mathcal{O}(t^{N+1}),$  $\alpha_i(t) - \alpha_i^N(t) = \mathcal{O}(t^{N+1}), i=\overline{1,n-1}.$
 
\end{enumerate}

Expansions in powers of $t, s$ in {\bf (A)} we call below as ``Taylor polynomials''
of the corresponding elements.
Let us introduce the matrix
\begin{equation}
B(j) = K_n(0,0) + \sum\limits_{i=1}^{n-1} (\alpha_i^{\prime}(0))^{1+j}(K_i(0,0)-K_{i+1}(0,0))
\label{eq4}
\end{equation}
and the following algebraic equation
$$L(j) \stackrel{\mathrm{def}}{=} \det B(j) = 0. $$
We call  it as {\it characteristic equation}
of the system of integral equations \eqref{eq1}.
Since $f(0)=0$ matrices $K_i(t,s)$ and vector $f(t)$
have continuous derivatives wrt $t,$
then differentiation of the system \eqref{eq1}
lead us to the equivalent system of integral-functional equations

\begin{equation}
F(x) \stackrel{\mathrm{def}}{=} K_n(t,t)x(t) +\sum\limits_{i=1}^{n-1}
\alpha_i^{\prime}(t)\bigl\{ K_i(t,\alpha_i(t)) - 
\label{eq5}
\end{equation}
$$-K_{i+1}(t,\alpha_i(t)) \bigr\}x(\alpha_i(t))+\sum\limits_{i=1}^n \int\limits_{\alpha_{i-1}(t)}^{\alpha_i(t) } K_i^{(1)}(t,s)x(s)ds - f^{\prime}(t)=0, $$
where $\alpha_0=0, \alpha_n(t)=t.$

Here we do not assume that homogenius system of  \eqref{eq1},
has only the trivial solution. Hence the homogeneous integral-functional system corresponding to 
 \eqref{eq5} can have nontrivial solutions. Here we follow \cite{lit2, lit12} 
and  seek an asymptotic approximation of a particular solution of the inhomogeneous
equation \eqref{eq5} as following polynomial
\begin{equation}
\hat{x}(t) = \sum\limits_{j=0}^N x_j (\ln t) t^j.
\label{eq6}
\end{equation}
Let us demonstrate that coefficients $x_j$ depend on $\ln t$ and free parameters
in general irregular case. That is in line with possibility of nontrivial solutions
existance of the  homogeneous system.

The regular and irregular cases are possible when compute the 
coefficients  $x_j.$

\subsection{Regular case: $L(j) \neq 0,$ $j\in (0,1,\hdots ,N)$}

In this case the coefficients $x_j$ are constant vectors from 
$\mathbb{R}^m.$ Indeed, let us substitute an expansion
 \eqref{eq6} in the system
\eqref{eq5} and apply the method of undetermined
coefficients taking into account the condition {\bf (A)}.
Then we get the recurrent sequence of linear  system of algebraic equations (SLAE) with respect to the vectors $x_j:$
\begin{equation}
B(0)x_0 = f^{\prime}(0),
\label{eq7}
\end{equation}

\begin{equation}
B(j)x_j = M_j(x_0, \dots , x_{j-1}), \, j=1,\dots , N.
\label{eq8}
\end{equation}
Vector $M_j$ can be expressed in a certain way through the solutions
$x_0, ... , x_{j-1}$ of the previous systems and coefficients 
of the ``Taylor polynomials'' from the condition {\bf (A)}.

Since in the regular case $\det B(j) \neq 0,$ then vectors
$x_0, \dots , x_N$ can be uniquely determined and asymptotic \eqref{eq6}
can be therefore constructed.

\subsection{Irregular case: equation $L(j) = 0$ has 
integer roots.}

Let us introduce the definitions:

\noindent {\bf Definition 1.}
We call quantity $j^*$ {\it the regular point of the matrix} $B(j)$
if matrix $B(j^*)$ is invertible.

\noindent {\bf Definition 2.}
We call quantity  $j^*$  {\it simple singular point of the matrix} $B(j)$
if  $\det B(j^*) = 0$,
$\det \bigl[(B^{(1)}(j^*)\phi_i, \psi_l) \bigr]_{i,l=1}^r \neq 0,$
where $\{\psi_i\}_1^r$ is basis
in $N(B^{\prime}(j^*)),$ $B^{\prime}(j^*)$ is transposed matrix, and
$B^{(1)}(j)$ is derivative of the matrix wrt $j.$



\noindent {\bf Definition 3.}
We call quantity $j^*$  $k+1$-multiple singular point of the matrix $B(j)$
if $\det B(j^*) = 0,$ derivatives
$B^{(1)}(j^*), \dots , B^{(k)}(j^*)$ are zero matrices,
$$\det \biggl[(B^{(k+1)}(j^*)\phi_i, \psi_l) \biggr]_{i,l=1}^r \neq 0,$$
 $k \geq 1,$ $\{\phi_i \}_1^r$ is basis in $N(B(j^*)),$
$\{\psi_i\}_1^r$ is basis in  $N(B^{\prime}(j^*)).$

Let us notice that
$B^{(k)}(j) = \sum\limits_{i=1}^{n-1} (\alpha_i^{\prime}(0))^{1+j} a_i^k (K_i(0,0) - K_{i+1}(0,0)),$
where $a_i = \ln \alpha_i^{\prime}(0).$

\begin{remark}
Let $m=1$ (this case correspond to the single equation \eqref{eq1}).
Then $B(j)=L(j).$ Hence in case of the single equation the definition 2 means that
$j$ is the single root of the characteristic equation $L(j)=0,$ 
and definition 3 means that  $j$ is $(k+1)$ multimple root of this equation.
\end{remark}

Let us demonstrate that in irregular case the coefficients $x_j$ are polynomials
in power of  $\ln t$ and depend on arbitrary constants. Order of polynomials
and number of arbitrary constants are associated with 
muliplicity of singular points of the matrices
$B(j)$ and ranks of these matrices.

Indeed, since the coefficient $x_0$ in irregular case may
depend on 
$\ln t,$ then we may employ the method
of undetermined coefficients and seek $x_0$ as solution of
the difference system
\begin{equation}
K_n(0,0) x_0(z) + \sum\limits_{i=1}^{n-1} \alpha_i^{\prime}(0) (K_i(0,0) - K_{i+1}(0,0)) x_0(z+a_i) =
f^{\prime} (0),
\label{eq9}
\end{equation}
where $a_i = \ln \alpha^{\prime}(0), z = \ln t.$

Here there are three cases:

\noindent {\it 1st Case.}\\
$L(0)\neq 0,$ i.e. $\det B(0) \neq 0.$ Then the coefficient $x_0$ does not depend on$z$
and determined uniquely from the SLAE \eqref{eq7}
with invertable matrix $B(0).$

\noindent {\it 2nd Case.}\\
Let $j=0$ be  the simple singular point of the matrix $B(j).$ We seek the coefficient $x_0(z)$
from the system \eqref{eq9}
as linear vector-function 
\begin{equation}
x_0(z) = x_{01} z + x_{02}.
\label{eq10}
\end{equation}
Let us substitute  \eqref{eq10} in \eqref{eq9}, and we get the following two
SLAE for determination of the vectors $x_{01}, x_{02}$: 


\begin{equation}
B(0) x_{01} = 0,
\label{eq11}
\end{equation}

\begin{equation}
B(0) x_{02} + B^{(1)}(0) x_{01} = f^{\prime}(0).
\label{eq12}
\end{equation}

Here $\det B(0) = 0,$ $\{\phi_i\}_1^r$ is basis in $N(B(0)).$
Hence $x_{01} = \sum\limits_{k=1}^r c_k \phi_k.$
Vector $c = (c_1, \dots , c_r)^{\prime}$
can be determined uniquely 
from the conditions of the system 
\eqref{eq12} resolvability, i.e. fro the SLAE
$$\sum\limits_{k=1}^r (B^{(1)}(0)\phi_k, \psi_i) c_k  = (f^{\prime}(0),\psi_i), \, i=\overline{1,r}$$
with non-singular matrix.
The coefficient $x_{02}$ can be determined from the system \eqref{eq12}
with accuracy up to $\mathrm{span}(\phi_1, \dots , \phi_r).$
Therefore, in 2nd case the coefficient $x_0(z)$
will depend on  $r$ arbitrary constants and linear wrt
$z.$\\

\noindent {\it 3rd Case. } Let $j=0$ be the singular point of the matrix $B(j)$ of multiplicity $k+1,$
where $k \geq 1.$ Solution $x_0(z)$ to the difference system \eqref{eq9}
we search as polynomial
\begin{equation}
x_0(z) = x_{01} z^{k+1} + x_{02} z^k + \dots + x_{0 k+1} z + x_{0 k+2}.
\label{eq13}
\end{equation} 
Let us substitute the polynomial \eqref{eq13}
into the system \eqref{eq9}. We next take into account the equality $$\frac{d^k}{dj^k}B(j) = \sum\limits_{i=1}^{n-1} (\alpha_i^{\prime}(0))^{1+j} a_i^k (K_i(0,0) - K_{i+1}(0,0)),$$
where $a_i = \ln \alpha_i^{\prime}(0)$
and equation  coefficients of powers $z^{k+1}, z^k, \hdots , z, z^0$ to zero
yields the recursive sequence of linear algebraic equation with respect to the
coefficients $x_{01}, x_{02}, \dots , x_{0 k+2}:$\\
$B(0) x_{01} = 0,$\\
$B(0) x_{02} + B^{(1)}(0) 
\biggl(\begin{array}{c}
	k+1\\ k
	\end{array}
	\biggr) x_{01} = 0,$\\
	$B(0) x_{0 l+1} + B^{(l)}(0) \biggl(\begin{array}{c}
	k+1\\ k+1-l
	\end{array}
	\biggr) x_{01} + B^{(l-1)}(0) \biggl(\begin{array}{c}
	k\\ k+1-l
	\end{array}
	\biggr) x_{02} + \hdots  \\ \hdots + B^{(1)}(0) 
	\biggl(\begin{array}{c}
	k+1-l+1\\ k+1-l
	\end{array}
	\biggr) x_{ol} = 0, \, l=1,\dots , k,
	$
\begin{equation}
B(0) x_{0 k+2} + B^{(k+1)}(0) x_{01} + B^{(k)} x_{02} + \dots B^{(1)}(0) x_{0 k+1} = f^{\prime}(0).
\label{eq14}
\end{equation}

Since in this case due to 
the conditions of the definition 3 the derivatives $\frac{d^iB(j)}{dj^i}\biggr|_{j=0},$
 $i=1,...,k $ are zero matrices, then 
$$x_{0i} = \sum\limits_{j=1}^r c_{ij} \phi_j, i=1,\dots, k+1. $$ 
As result the system \eqref{eq14} can be rewritten as
\begin{equation}
B(0)x_{0,k+2} + B^{(k+1)}(0) x_{01} = f^{\prime}(0).
\label{eq15}
\end{equation}
Since $\det \biggl[ (B^{(k+1)}(0)\phi_i, \psi_k) \biggr]_{i,k=\overline{1,r}} \neq 0,$
then vector $c^1 \stackrel{\mathrm{def}}{=} (c_{11}, \dots , c_{1r})^{\prime} $
can be uniquly determined from
the conditions of resolvability of the system \eqref{eq15}.
Therefore, 
$$ x_{0\, k+2} = \sum\limits_{j=1}^r c_{{k+2}\, j} \phi_j + \hat{x}_{k+2}, $$
$\hat{x}_{k+2}$ is the particular solution to the SLAE \eqref{eq15}.
 Vector $c^{k+2} \stackrel{\mathrm{def}}{=}  (c_{k+2,1}, \dots c_{k+2,r})^{\prime},$
and vectors $c^2, \dots , c^{k+1},$
remain arbitrary. Therefore in 3rd case the coefficient $x_0(z)$
is polynomial of power of  $k+1$ wrt $z$ and it depends on
$r \times (k+1) $ arbitrary constants.
One may apply the method of undetermined 
coefficients and take into account the following equality
$$\int t^j \ln^k t dt = t^{j+1} \sum\limits_{s=0}^k (-1)^s \frac{k(k-1) \dots (k-(s-1))}{(j+1)^{s+1}} \ln^{k-s} t, $$
and construct the system of difference
equations for determination of the coefficent
$x_1(z)$  ($z=\ln t$).
Indeed,
\begin{equation}
F(x)\biggr |_{x=x_0+x_1 t} \stackrel{\mathrm{def}}{=} \biggl[ K_n(0,0)x_1(z) + \sum\limits_{i=1}^{n-1}
(\alpha_i^{\prime}(0))^2 (K_i(0,0) -
\label{eq16}
\end{equation} 
$$K_{i+1}(0,0)) x_1(z+a_i) + P_1(x_0(z)) \biggr] t + r(t), \,\,\,\,\, r(t) = o(t).$$
Here $P_1(x_0(z))$ is certain polynomial on $z,$ which power is equal to the multiplicity
of the singular point $j=0$ of the matrix $B(j).$
From \eqref{eq16} due to the estimate $r(t)=o(t)$
for $t\rightarrow 0$ it follows that the coefficient $x_1(z)$
must satisfy the the following system of difference equations

\begin{equation}
K_n(0,0) x_1(z) + \sum\limits_{i=1}^{n-1} (\alpha^{\prime}(0))^2 \bigl(K_i(0,0) - K_{i+1}(0,0)\bigr) x_1(z+a_i)+
\label{eq17}
\end{equation}
$$+ P_1(x_0(z)) = 0. $$
If $j=1$ is regular point of the matrix $B(j),$ then system \eqref{eq17}
has solution $x_1(z)$ as polynomial of the same order as
multiplicity of the singular point $j=0$
of the matrix $B(0).$
If $j=1$ is singular point of the matrix $B(j),$ then solution $x_1(z)$
can be constructed as polynomial of power  $k_0 +k_1,$
where $k_0$ and $k_1$ are multiplicity of singular point $j=0$ and $j=1$
of the matrix $B(j)$ correspondingly. The coefficient $x_1(z)$  depends on 
$r_0 k_0 + r_1 k_1$
arbitrary constants, where $r_0 = \mathrm{dim} N(B(0)),$
$r_1 = \mathrm{dim} N(B(1)).$

Let us impose the following  condition
\begin{enumerate} [{\bf (B)}]
\item
Let matrix $B(j)$ 
has only the regular points in the array $(0,1, \dots , N)$ or has singular points 
of multiplicities $k_j.$
\end{enumerate}

 Then, in a similar way we can calculate the remaining coefficients
$x_2(z), \dots , x_N(z)$ of $\hat{x}(t)$
from the sequence of difference equations
$$K_n(0,0)x_j(z) + \sum\limits_{i=1}^{n-1} (\alpha^{\prime}(0))^{1+j}\bigl(K_i(0,0)-K_{i+1}(0,0)\bigr)x_j(z+a_i)+$$
$$+\mathcal{P}_j(x_0(z), \dots , x_{j-1}(z)))=0, \, j=\overline{2,N}.$$

Hence we have the following\\
\begin{lemma}
Let conditions {\bf (A)}, {\bf (B)} be fuilfilled.  Then exists vector function
$\hat{x}(t) = \sum\limits_{i=0}^N x_i (\ln t) t^i,$ such as for $t \rightarrow +0$
the following estimate is satisfied
$|F(\hat{x}(t))|_{\mathbb{R}^m} = o(t^N)$
and the coefficients $x_i (\ln t)$ are polynomials of  $\ln t$
with insreasing powers smaller then sum $\sum\limits
_j k_{j}$ of the singular points $j$ of the matrix $B(j)$ from the array $(0,1,\dots , i).$
Coefficients $x_i(\ln t)$ depend on $\sum\limits_{j=0}^i \mathrm{dim} N(B(j)) k_j$
arbitrary constants.
\end{lemma}

\section{Existence of continuous parameter family of solutions theorem}

Since $0 \leq \alpha_i^{\prime}(0) <1, \, \alpha_i(0)=0, \, i=\overline{1,n-1},$
then for any $0<\varepsilon <1$
exists $T^{\prime} \in(0,T]$
such as  $\max\limits_{i=\overline{1,n-1}, t\in [0,T^{\prime}]} |\alpha_i^{\prime} (t)| \leq \varepsilon$
and $\sup\limits_{i=\overline{1,n-1}, t\in (0,T^{\prime}]} \frac{\alpha_i(t)}{t} \leq \varepsilon.$

Let the following condition be fulfilled

\begin{enumerate} [{\bf (C)}]
\item Let $\det K_n(t,t) \neq 0,$ $t\in [0,T^{\prime}]$ and $N^*$
is selected large enough to have the following inequality be satisfied
\begin{equation}
\max\limits_{t\in [0,T]} \varepsilon^{N^*} |K_n^{-1}(t,t)|_{\mathcal{L}(\mathbb{R}^m \rightarrow \mathbb{R}^m)} \times
\end{equation}
$$\times \sum\limits_{i=1}^{n-1} |\alpha_i^{(1)}(t)| |K_i(t,\alpha_i(t)) - K_{i+1}(t,\alpha_i(t))| _{\mathcal{L}(\mathbb{R}^m \rightarrow \mathbb{R}^m)}\leq q <1,
$$
 where $|\cdot|_{\mathcal{L}(\mathbb{R}^m \rightarrow \mathbb{R}^m)} $ is norm of the 
 $m \times m$ matrix. 
\end{enumerate}
 
\begin{lemma}
Let the condition  {\bf (C)} be fulfilled.
Let $\mathbb{C}_{(0,T^{\prime}]}$ be class
of the vector functions continuous for $t \in (0,T^{\prime}]$
and and having a limit (possibly infinite) in $t \rightarrow +0.$
 Let in  $\mathbb{C}_{(0,T^{\prime}]}$
exists the element $\hat{x}(t)$ such as the following estimate
$$|F(\hat{x}(t))|_{\mathbb{R}^m} = o(t^N), N \geq N^*$$
is fulfilled. 
Then equation \eqref{eq5} in $\mathbb{C}_{(0,T^{\prime}]}$
has the solution
\begin{equation}
x(t) = \hat{x}(t) + t^{N^*} u(t),
\label{eq19}
\end{equation}
where $u(t)\in \mathbb{C}_{[0,T^{\prime}]}$
and can be uniquely defined with successive approximations.
\end{lemma}

\begin{proof}
Let us substitute \eqref{eq19} in the equation \eqref{eq5}.
We get the following integral-functional system for determination of $u(t)$ 
\begin{equation}
K_n(t,t) u(t) + \sum\limits_{i=1}^{n-1} \alpha_i^{\prime}(t)
\biggl(\frac{\alpha_i(t)}{t} \biggr)^{N^*}  \biggl( K_i(t, \alpha_i(t))  -
\label{eq20}
\end{equation}
$$
-K_{i+1}(t,\alpha_i(t)) \biggr) u(\alpha_i(t)) +$$ $$+\sum\limits_{i=1}^n
\int\limits_{\alpha_{i-1}(t)}^{\alpha_{i}(t)} K_i^{(1)}(t,s) \biggl(\frac{s}{t}\biggr)^{N^*} u(s) ds
+ F(\hat{x}(t)) / t^{N^*} = 0.
$$
Let us introduce the linear opearators:
$$
Lu \stackrel{\mathrm{def}}{=}  K_n^{-1}(t,t)  \sum\limits_{i=1}^{n-1}
\alpha_i^{\prime}(t) \biggl(\frac{\alpha_i(t)}{t} \biggr)^{N^*}
\biggl \{ K_i(t, \alpha_i(t))  -$$
$$-K_{i+1}(t,\alpha_i(t)) \biggr\} u (\alpha_i(t)), $$
$$K u \stackrel{\mathrm{def}}{=} \sum\limits_{i=1}^n \int\limits_{\alpha_{i-1}(t)}^{\alpha_i(t)}
K_n^{-1}(t,t) K_i^{(1)}(t,s) (s/t)^{N^*} u(s) ds. $$
Then we can present the system \eqref{eq20} as follows
$$u+ (L+K) u  = \gamma(t), $$
where $\gamma(t) = K_n^{-1}(t,t) F(x^N(t))/t^{N^*}$ is continuous 
vector function.
Let us introduce banach space $X$ of continuous vector-functions $u(t)$
with norm $$ ||u||_l = \max\limits_{0\leq t \leq T^{\prime}} e^{-lt} |u(t)|_{\mathbb{R}^m}, \, l>0.$$
Hence because of  inequalities $\sup\limits_{t \in [0,T]} \frac{\alpha_i(t)}{t} \leq \varepsilon <1$
and condition {\bf C)} for $\forall l \geq 0$ 
norm of linear functional operator  $L$ satisfies the following estimate
$$||L ||_{\mathcal{L}(X \rightarrow X)} \leq q < 1.$$
Moreover for integral operator  $K$
for large enough $l$ the following estimate is fulfilled
$$||K ||_{\mathcal{L}(X \rightarrow X)} \leq q_1 < 1-q.$$
As result for large enough $l>0$ we have
$$||L+K ||_{\mathcal{L}(X \rightarrow X)} < 1,$$
i.e. the linear operator $L+K$
is contracting in  $X.$
Hence the sequence 
 $u_n = -(L+K)u_{n-1} + \gamma(t), \, u_0=\gamma(t)$
converges.

\end{proof}

\begin{theorem}{\it (Main Theorem).}
Let the following conditions {\bf (A)}, {\bf (B)}, {\bf (C)}, $f(0)=0$ be fulfilled. 
Let also the matrix $B(j)$ has exactly $\nu$ singular points $j_1, \dots , j_{\nu}$
of multiplicities $k_i,$ $i=\overline{1,\nu},$ 
in the array $(0,1,\dots , N))$ and the rest of the points of the array are regular.
Let $rank B(j_i) = r_i,$ $i=\overline{1,\nu}$. 

Then equation \eqref{eq1} for $0<t\leq T^{\prime} \leq T$ has he solution
$$ x(t) = \hat{x}(t) + t^{N^*} u(t), $$
depending on  $\sum\limits_{i=1}^{\nu} (m-r_i) k_i$
arbitrary constants.

\end{theorem}

\begin{proof}
Due to the conditions of the Lemma 1 and imposed conditions of the theorem
it make it possible the construction of asymptotic
approximation
$\hat{x}(t)$ of the desired solution in the from of logarithmic power polynomial $\sum\limits_{i=0}^N x_i (\ln t) t^i$.
Moreover the coefficients $x_i(\ln t)$  will depend on the specified number of
arbitrary constants.
Because of the Lemma 2 we can apply the substitution $x(t) = \hat{x}(t)+
t^{N^*} u(t),$ and continuous function $u(t)$ 
can be constructed with successive approximations.

\end{proof}

\begin{remark}
In the main theorem's conditions for the asymptotic approximation $\hat{x}(t)$ of the desired solution the following asymptotic estimate
$|x(t) - \hat{x}(t)|_{\mathbb{R}^n} = \mathcal{O}(t^{N^*}), \, t\rightarrow +0$
is fulfilled.
\end{remark}

\begin{corollary}
Let $\alpha_i(t) = \alpha_i t,$  $i=\overline{1,n-1},$  $0< \alpha_1 < \alpha_2 < \hdots < \alpha_{n-1} <1, $
elements of the matrices $K_i(t,s),$  $i=\overline{1,n-1}$ and vector function $f(t)$
has an analytic extension to $|s| < T, \, |t| < T,$ $ f(0)=0.$
 Matrix $K_n(t,t)^{-1}$ is analytical for $|t|<T.$
Let $\det B(j) \neq 0,$ $j\in \mathbb{N} \cup {0}. $
Then equation \eqref{eq1} has unique solution $x(t) = \sum\limits_{i=0}^{\infty} x_i t^i$ for
$0\leq t < T.$
\end{corollary}

In some cases, the conditions of Theorem 1 we can construct a parametric
family of solutions in closed form.

\noindent {\bf Example.}
Let us consider the system $\int\limits_0^{t/2} K x(s) ds + \int\limits_{t/2}^t (K- 2E) x(s) ds = d t,$
$0<t < \infty,$
where $K$ is symmetric constant matrix $m \times m, \, d \in \mathbb{R}^m,$
$x(t) = (x_1(t), ...,x_m(t) )^{\prime},$
$1$ is eigenvalue of the matrix  $K$ of rank $r,$
$\{ \phi_1, ... , \phi_r \}$ is corresponding orthonormal system
of eigenvectors. This system has parametric family of solutions
$$x(t) = -\ln t \sum\limits_{i=1}^r \frac{(d, \phi_i)}{\ln 2}\phi_i + c_1 \phi_1 + ... + c_r \phi_r + \hat{a}.$$
Here $c_1, ..., c_r$ are arbitrary constants, vector $\hat{a}$
satisfies the SLAE
$$(K-E)\hat{a} = d - \sum\limits_{i=1}^r (d, \phi_i) \phi_i. $$


\section{Sufficient conditions for existence of a unique continuous solution}

For sake of clarity in this section let us suppose $\alpha_i(t) = \alpha_i t, \, i=\overline{1,n-1}, \, 0<\alpha_1 < \alpha_2 < ... < \alpha_{n-1} < 1.$
Let us introduce the matrix $$D(t) = \sum\limits_{i=1}^{n-1} \alpha_i K_n^{-1}(t,t)(K_i(t,\alpha_1 t) - K_{i+1}(t,\alpha_1 t)). $$
Let the following condition be fulfilled:

\begin{enumerate} [{\bf (S)}]
\item  $|D(0)|_{\mathcal{L}(\mathbb{R}^m \rightarrow \mathbb{R}^m)} < 1;\,$
$\sup\limits_{0 \leq s \leq t \leq T}|K_n^{-1}(t,t) K(t,s)|_{\mathcal{L}(\mathbb{R}^m \rightarrow \mathbb{R}^m)} \leq c<\infty. $
\end{enumerate}
Here and below the matrix  $K(t,s)$ is defined with formula \eqref{eq2}.

\begin{theorem} { (Sufficient conditions for the existence and uniqueness of solutions).}
Let the conditions {\bf (S)} be fulfilled, all the matrices $K_i(t,s)$ are continuous in \eqref{eq2}
and they are continuously differentiable wrt $t,$ vector $f(t)$
has continuous derivative, $f(0)=0.$
Then an equation \eqref{eq1} has unique solution in  $\mathbb{C}_{[0,T]}.$ 
Moreover, the solution can be found by the method of steps, combining it with the method of successive approximations.
\end{theorem}

\begin{proof}
Let us rewrite the equation \eqref{eq5} which equivalent to the equation \eqref{eq1}
as follows
\begin{equation}
x(t) + A x + K x = \overline{f}(t),
\label{eq21}
\end{equation}
where the following notations are introduced\\
$Ax \stackrel{\mathrm{def}}{=} K_n^{-1}(t,t) \sum\limits_{i=1}^{n-1} \alpha_i (K_i(t,\alpha_i t) - K_{i+1}(t,\alpha_i t)) x(\alpha_i t),$\\
$Kx \stackrel{\mathrm{def}}{=} \sum\limits_{i=1}^n \int\limits_{\alpha_{i-1}t}^{\alpha_i t} K_n^{-1}(t,t) K_t^{(1)}(t,s) x(s)ds,\,\,$
$\overline{f}(t) = K_n^{-1}(t,t) f^{(1)}(t).$

Let us fix $q<1$ and select $h_1>0$ such as $\max\limits_{0\leq t \leq h_1}|D(t)|_{\mathcal{L}(\mathbb{R}^m \rightarrow \mathbb{R}^m)} = q<1.$
Because of the condition {\bf (S)} exists such $h_1>0$.
Let $0<h<\min \{ h_1, \frac{1-q}{c} \},$
where constant $c$ is defined in the condition {\bf (S)}.
Lets divide the interval $ [0, T] $ into the intervals
\begin{equation}
[0,h], \, [h, h+\varepsilon h], \, [h+\varepsilon h, h+ 2\varepsilon h], \hdots
\label{eq22}
\end{equation}
where $\varepsilon$ is selected from $(0,1]$
such as  $\alpha_{n-1} \leq \frac{1}{1+\varepsilon}.$
Denote by $x_0(t)$ the restriction of the desired solution $ x(t) $
on the interval $[0,h]$ and let $x_n(t)$ be the restrictions to intervals
$$I_n = [(1+(n-1)\varepsilon)h, (1+n\varepsilon)h], n=1,2,\dots.$$
Because of the selection of  $\varepsilon$ for $t\in I_n$
the ``perturbed'' argument $\alpha_i t \in \bigcup\limits_{k=1}^{n-1} I_k.$
Such an inclusion makes it possible to apply the known method of steps \cite{lit_elsgoltz}.

For construction of the element $x_0(t) \in \mathbb{C}_{[0,h]}$
we construct the sequence $\{x_0^n(t)\}:$
$$x_0^n (t) = -A x_0^{n-1} - K x_0^{n-1} + \overline{f} (t), $$
$$x_0^0(t) = \overline{f}(t), \, t \in [0,h]. $$
Because of the selection of  $h$ we have an estimate $||A+K||_{\mathcal{L}(\mathbb{C}_{[0,h]} \rightarrow \mathbb{C}_{[0,h]})} < 1.$

Hence exists the unique solution  $x_0(t)$ of the equation \eqref{eq21} for $t\in [0,h]$. Sequence 
$x_0^n(t)$ converges uniformly to that  unique solution  $x_0(t)$.
Let us continue the process for the desired solution
construction for  $t\geq h,$ i.e. on the intervals
$I_n, \, n=1,2,\hdots .$
To be specific let $\varepsilon = 1$ in  \eqref{eq22}.

Once we get $x_0(t) \in \mathbb{C}_{[0,h]}$ computed we constract $x_1(t) $ 
in the space $ \mathbb{C}_{[h,2h]}$.
We find $x_1(t)$ from the Volterra integral equation of the 2nd kind
$$x(t) + \int\limits_h^t K_n^{-1}(t,t) K_t^{\prime}(t,s) x(s) ds = \overline{f}(t) - A x_0 - \int\limits_{0}^h K_n^{-1}(t,t)
K_t^{\prime}(t,s) {x}_0(s) ds$$
using successive approximations. In this case $x_0(h) = x_1(h).$

Let us introduce the continuous function
 \begin{equation}
    \overline{x}_1(t) = \left\{ \begin{array}{ll}
         \mbox{$x_0(t), \,\,\,\, {0} \leq t \leq h$}, \\
         \mbox{$x_1(t), \,\, h \leq t \leq 2h $}, \\
\end{array}          
\right.
 \end{equation}
which is Is the restriction of the desired continuous solution $ x(t) $ on the interval
$[0,2h].$
Then element $x_2(t) \in \mathbb{C}_{[2h, 3h]}$
can be computed with successive approximations
from the Volterra integral equation of the second kind
$$x(t) + \int\limits_{2h}^t K_n^{-1}(t,t) K_t^{\prime}(t,s) x(s) ds = \overline{f}(t) - A \overline{x}_1
- \int\limits_{0}^{2h} K_n^{-1}(t,t) K_t^{\prime}(t,s) \overline{x}_1(s) ds.$$

Finally we constract the desired solution
$x(t) \in \mathbb{C}_{[0,T]}$ to the equation \eqref{eq1} during the $N$ steps ($N \geq \frac{T}{h}$).

\end{proof}

\noindent {\bf Example.}
Integral equation $$  \int\limits_{0}^{t/2} K_1(t-s) x(s) ds + \int\limits_{t/2}^{t} K_2(t-s) x(s) ds = f(t), \, 0< t \leq T, $$
where $K_1(t-s) = K_2(t-s) +E, $ are matrices $m \times m,$ $E$ is unit matrix,
$|K_2^{-1}(0)|_{\mathcal{L}(\mathbb{R}^m \rightarrow \mathbb{R}^m)}<2,$  matrix $K_2(t)$ 
and vector function $f(t)$
have continuous derivatives wrt $t,$ $f(0)=0,$
satisfies the Theorem 2 conditions and has the unique 
continuous solution.

\end{document}